\documentclass[12pt]{amsart}   
\usepackage{fullpage}

\usepackage{graphicx}
\usepackage{tikz}			
\usepackage{amssymb}
\usepackage{mathtools}
\usepackage{xcolor}
\usepackage{array}
\usepackage{mathabx}
\usepackage{amsmath}
\usepackage{amssymb}
\usepackage{amsthm}
\usepackage{url}
\usepackage{verbatim}
\usepackage[capitalize]{cleveref}
\usepackage{tkz-euclide}

\newtheorem{theorem}{Theorem}[section]

\newtheorem{lemma}[theorem]{Lemma}
\newtheorem{corollary}[theorem]{Corollary}

\newtheorem{proposition}[theorem]{Proposition}

\newtheorem{claim}{Claim}[theorem]

\newtheorem*{main}{Main Theorem}
\newtheorem*{do_thm}{Do's Theorem}
\newtheorem*{beck}{Beck's Theorem for Hyperplanes}
\newtheorem*{dirac}{Weak Dirac Theorem}

\newcommand{\del}{\setminus}
\newcommand{\con}{/}

\DeclareMathOperator{\cl}{cl}

\newcommand{\cF}{\mathcal{F}}
\newcommand{\cH}{\mathcal{H}}

\newcommand{\cP}{\mathcal{P}}

\newcommand{\cI}{\mathcal{I}}

\title[Characterizing real-representable matroids with large]{Characterizing real-representable matroids with large average hyperplane-size}
\author{Rutger Campbell}
\address{Discrete Mathematics Group, Institute for Basic Science (IBS), Daejeon, Republic of Korea} 
\author{Matthew E. Kroeker}
\address{Institute of Discrete Mathematics and Algebra, Technische Universit\"at Bergakademie Freiberg, Freiberg, Germany}
\author{Ben Lund}
\address{Institute of Mathematics and Interdisciplinary Studies, Xidian University, Xi’an, 710071, China.}

\thanks{Much of the research for this paper took place while the second author was hosted by the Discrete Mathematics Group (DIMAG), Institute for Basic Science (IBS) in Daejeon, Republic of Korea, in January and July-August 2024. M. Kroeker was supported by an Ontario Graduate Scholarship. R. Campbell and B. Lund were supported by the Institute for Basic Science (IBS-R029-C1).}

\begin{document}

\begin{abstract}
Generalizing a theorem of the first two authors and Geelen for planes, we show that, for a real-representable matroid $M$, either the average hyperplane-size in $M$ is at most a constant depending only on its rank, or each hyperplane of $M$ contains one of a set of at most $r(M)-2$ lines. Additionally, in the latter case, the ground set of $M$ has a partition $(E_{1}, E_{2})$, where $E_{1}$ can be covered by few flats of relatively low rank and $|E_{2}|$ is bounded. These results extend to complex-representable and orientable matroids. Finally, we formulate a high-dimensional generalization of a classic problem of Motzkin, Gr\"unbaum, Erd\H{o}s and Purdy on sets of red and blue points in the plane with no monochromatic blue line.
We show that the solution to this problem gives a tight upper bound on $|E_{2}|$. We also discuss this high-dimensional problem in its own right, and prove some initial results. 
\end{abstract}

\maketitle

\section{Introduction}

In 1943, Melchior \cite{Melchior} proved that, for a real-representable matroid of rank at least three, the average size of a line is strictly less than three. 
Here the {\it{size}} of a line (or flat of any rank) is its number of points. 
The direct sum of a point and a line demonstrates that Melchior's bound is best-possible. 
For real-representable matroids of higher rank, the average size of a hyperplane need not be bounded. For example, the average plane-size in the direct sum of two lines can be arbitrarily large, since every plane will contain one of the two large lines. However, for planes, this is essentially the only thing that can go wrong, as recently proved by the first two authors with Geelen:

\begin{theorem}[\cite{CGK}]\label{planes}
In a real-representable matroid of rank at least four which is not the direct sum of two lines, the average plane-size is at most an absolute constant. 
\end{theorem}

The idea of the direct sum of two lines generalizes to higher ranks following the definition of a $k$-degenerate matroid, introduced independently by Lov\'asz \cite{Lovasz} in the context of rigidity theory, and by Lund \cite{Lund} and Do \cite{Do} in incidence geometry.
We call a matroid $M$ $k$-{\it{degenerate}} if it either has rank at most one, or there is a collection $(F_{1}, \dots, F_{t})$ of flats of $M$, each having rank at least two, such that $E(M) \subseteq F_{1} \cup \dots \cup F_{t}$ and $\sum_{i=1}^{t} (r_{M}(F_{i})-1) \leq k-1$.
We may assume that the flats in $(F_{1}, \dots, F_{t})$ are pairwise-skew, as otherwise we can replace a non-skew pair with the span of their union.
We call a matroid $1$-{\it{degenerate}} if and only if it is the empty matroid.
Observe that a matroid is $2$-degenerate if and only if it has rank at most two, and it is $3$-degenerate if and only if it is either a direct sum of two lines or has rank at most three.
We call a subset $X \subseteq E(M)$ $k$-{\it{degenerate}} (in $M$) if $M|X$ is a $k$-degenerate matroid.

For a $k$-degenerate matroid $M$, it follows by the pigeonhole principle that every hyperplane $H$ of $M$ must contain one of the flats $F_{i}$.
Hence, the average hyperplane-size can be arbitrarily large in a degenerate matroid. In light of Theorem \ref{planes}, one optimistically might expect that forbidding $k$-degeneracy is sufficient for bounding the average size of a rank-$k$ flat by an absolute constant. A construction given in \cite{CGK} shows that this is not true for $k \geq 5$. 
However, it was shown that the average hyperplane-size in a real-representable matroid is bounded unless it is ``nearly degenerate'', in the sense that all but a constant number of elements are contained in a degenerate restriction.
It remained an open problem to fully characterize the real-representable matroids having average hyperplane-size at most an absolute constant. 

In this paper we show that, in a real-representable matroid $M$ with sufficiently many points, the average hyperplane-size is greater than an absolute constant only if there is some collection of at most $r(M)-2$ lines such that every hyperplane contains one of them.
Furthermore, we give a structural characterization of matroids with this property, which depends on a surprising connection with a classical problem posed independently by Motzkin \cite{motzkin1951lines}, Gr\"unbaum \cite{Grunbaum}, and by  Erd\H{o}s and Purdy \cite{ErdosPurdy}.
Their question was: Given $n$ blue and $t$ red points in the real affine plane, such that every monochromatic line is red, and the blue points are not all collinear, how large can $n$ be as a function of $t$?

To state our main result, we need to consider a higher-dimensional generalization of this question.
A $2$-{\it{colouring}} of a simple matroid $M$ is a partition $(R,B)$ of its ground set into `red' and `blue' sets. For integers $k,t \geq 1$, we define $b_{k}(t)$ to be the largest number $n$ for which there exists a simple rank-$(k+1)$ real-representable matroid $M$ having a $2$-colouring $(R,B)$ with $|R|=t$, $|B| = n$ such that
\begin{enumerate}
    \item[(C1)] $r_{M}(B) \geq 2$;
    \item[(C2)] for every hyperplane $H$ of $M$, $H \cap R \neq \emptyset$; and
    \item[(C3)] for every line $L$ of $M|B$, there is a rank-$(k-1)$ flat $F$ of $M$ skew to $L$ such that $F \cap R = \emptyset$.
\end{enumerate}
If, for given $k,t \geq 1$, there does not exist a number $n$ for which the above is satisfied, then we set $b_{k}(t)=0$. We will prove that the number $b_{k}(t)$ is indeed well-defined for all integers $k,t \geq 1$ (see Theorem \ref{bk_bound}).
For $k=1$, each point is a hyperplane, so (C2) implies that $b_{1}(t)=0$ for all $t$.
Also note that, for $k=2$, (C3) is equivalent to the condition that the set of blue points is not collinear, and so $b_{2}(t)$ is precisely the answer to the question of Motzkin, Gr\"unbaum, Erd\H{o}s, and Purdy. 

\begin{main}\label{th:main}
    For $k \geq 2$, if $M$ is a rank-$(k+1)$ real-representable matroid, then either 
    \begin{enumerate}
        \item the average size of a hyperplane of $M$ is at most a constant $C_{k}$ depending only on $k$, or
        \item there is a collection $L_1,\ldots,L_t$, where $t \leq k-1$, of lines of $M$ such that each hyperplane of $M$ contains one of them.
    \end{enumerate} 
    Moreover, if the average hyperplane-size in $M$ is greater than $C_{k}$, then for some $j \in \{1, \dots, k-1\}$, $M$ has a $(j+1)$-degenerate restriction containing all but at most $b_{k-j}(j)$ of its points.
\end{main}

For a matroid with large average hyperplane-size, the above lower-bound on the size of a degenerate set is best-possible: in Section 1.1, we will show how to construct real-representable matroids with the structure defined in alternative (2) whose ground set consists of a $(j+1)$-degenerate set together with exactly $b_{k-j}(t)$ additional points. 

In Section 2, we will show that $b_k(2)=0$ for all $k$, which means that alternative (2) is only possible for a matroid of rank at most five if it is degenerate.
Moreover, it is not hard to show that $b_2(3)=4$. 
This yields a reasonably precise characterization of the real-representable matroids of rank at most six with large average hyperplane-size.

\begin{corollary}\label{cor:3456}
Let $M$ be a simple real-representable matroid.
If $r(M) \leq 5$, then either $M$ is $(r(M)-1)$-degenerate, or the average hyperplane-size in $M$ is at most an absolute constant.
If $r(M)=6$ and $M$ is not $5$-degenerate, then either $M$ has a $4$-degenerate restriction of size at least $|E(M)|-4$, or the average hyperplane-size in $M$ is at most an absolute constant.
\end{corollary}

It is well known that Melchior's theorem extends to the classes of complex-representable and orientable matroids. In the latter case, the result extends by direct translation of Melchior's original proof to the setting pseudoline arrangements. On the other hand, it follows by a deep theorem of Hirzebruch \cite{hirzebruch1983arrangements} that the average line-size in a rank-$3$ complex-representable matroid is at most four.
While the results of this paper are formulated for real-representable matroids, they do in fact generalize to both the complex and orientable settings. 
The reason for this is that the proof of the Main Theorem is a purely matroidal inductive argument on $k \geq 2$, where the base case is equivalent to a classical theorem of Beck \cite{Beck}, which says that the number of lines in a real-representable matroid on $n$ points is $\Theta(n(n-\ell))$, where $\ell$ is the size of a largest line. 
While Beck's theorem fails for matroids in general, it does hold for complex-representable and orientable matroids, as explained below.

For $\gamma >0$, we call a class $\mathcal{M}$ of matroids $\gamma$-{\it{Beck}} if, for every matroid $M$ in $\mathcal{M}$ on $n$ points, the following are satisfied: 
\begin{enumerate}
    \item[(B1)] every minor of $M$ is in $\mathcal{M}$;
    \item[(B2)] every free extension of $M$ in any flat is in $\mathcal{M}$; and
    \item[(B3)] $M$ has at least $\gamma n(n-\ell)$ lines, where $\ell$ is the size of a largest line in $M$.
\end{enumerate}
Given a matroid in a $\gamma$-Beck class, every assertion in our proof is true up to the specification of relevant multiplicative constants (which depend on $\gamma$). In particular, because Beck's Theorem is weaker than the Szemer\'edi-Trotter Theorem \cite{SzemerediTrotter}, which was proved for complex-representable matroids independently by T\'oth \cite{Toth} and Zahl \cite{Zahl}, the Main Theorem also holds for this class. Similarly, Sz\'ekely \cite{szekely1997crossing} gave a proof of the Szemer\'edi-Trotter theorem that is easily seen to bound the number of incidences between points and pseudolines, and hence the Main Theorem also holds for the orientable matroids.

We conclude this section by presenting some further context for the Main Theorem and its proof. First we discuss the problem of Motzkin-Gr\"unbaum-Erd\H{o}s-Purdy and its higher-dimensional generalization in greater detail. After that, we introduce the key problem to which we reduce the proof of the Main Theorem, namely the problem of {\it{counting hyperplanes}}. In Section 2, we give several bounds on the number $b_{k}(t)$, and ultimately show that $b_{k}(t) = \Theta_{k}(t)$. In Section 3, we prove some helpful lemmas on the operation of {\it{principal truncation}}, which we use to invoke the inductive hypothesis in our proof. In Section 4, we prove the Main Theorem.

\subsection{The problem of Motzkin, Gr\"unbaum, Erd\H{o}s and Purdy}

Answering a question of Graham (see \cite{grunbaum1999monochromatic}), Motzkin and Rabin \cite{motzkin1967nonmixed} and Chakerian \cite{chakerian1970sylvester} showed that any set of non-collinear red and blue points in the real plane determines a monochromatic line.
Various generalizations of the Motzkin-Rabin theorem to higher dimensions have been investigated \cite{dvir2015quantitative, huicochea2021number, pretorius2009sylvester}.

If all monochromatic lines in such an arrangement are red, then one would generally expect that the number of blue points does not greatly exceed the number of red ones. This prompted Motzkin \cite{motzkin1951lines}, Gr\"unbaum \cite{Grunbaum}, and Erd\H{o}s and Purdy \cite{ErdosPurdy} to pose the problem defined above on bounding $b_{2}(t)$.

The earliest result on this problem was by Motzkin \cite{motzkin1951lines}, who showed that $b_2(t) \leq (t-1)^2$.
In introducing the problem, Motzkin noted that there is no obvious generalization of this result to higher dimensions.

Currently, the best general upper bound on $b_2(t)$ is due to Huicochea, Lea{\~n}os, and Rivera \cite{huicochea2022note}, who built on the work of Pinchasi \cite{pinchasi2013solution} to show that $b_2(t) \leq (11/4)t + 13/4$.
For lower bounds, Gr\"unbaum found families of point sets showing that $b_2(t) \geq t+4$ for infinitely many $t$, and a construction that is not known to be a member of any infinite family showing that $b_2(10) \geq 16$ \cite{grunbaum1999monochromatic}.

More precise answers are known for some special types of arrangements.
The most famous of these is Ungar's proof \cite{ungar1982directions} of Scott's conjecture \cite{scott1970sets} that $n$ points in the plane  not all contained in a single line determine at least $2\lfloor n/2 \rfloor$ directions.
This corresponds to bichromatic point sets without monochromatic blue lines in which the red points are required to be collinear (since each direction corresponds to a point on the line at infinity in the projective plane).
Erd\H{o}s and Purdy \cite{ErdosPurdy} also asked for the largest number of blue points with no three collinear, all of whose connecting lines can be covered by $n$ red points.
The answer to this question is exactly $n$, as was first proved by Ackerman, Buchin, Knauer, Pinchasi and Rote \cite{ackerman2008there}. Simpler proofs were later found by Pinchasi and Polyanskii \cite{pinchasi2021theorem}, \cite{pinchasi2020one}.

Recall from the Main Theorem that a simple real-representable matroid $M$ with large average hyperplane-size has a $j$-degenerate set of size at least $|E(M)| - b_{r(M)-j}(t)$ for some $j \in \{1, \dots, r(M)-1\}$ and $t \leq r(M)-2$. We now give a construction which demonstrates that this bound is best-possible. 
As a special case, consider the bichromatic set of points in \cref{fig:b2(3)}.
Adjoin to each red point an $m$-point line in a new direction so as to obtain a set of $3m+4$ points in $\mathbb{R}^{5}$. (Recall that $b_{2}(3)=4$.) Any hyperplane spanned by this resulting point set must contain one of the $m$-point lines. The reason for this is that every line in the arrangement in \cref{fig:b2(3)} contains a red point.

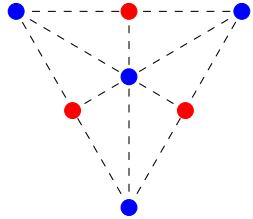
\begin{figure}[h!]
\begin{center}
\begin{tikzpicture}[scale=3]

\draw[dashed] (0,0)--(-1/2,.87);
\draw[dashed] (0,0)--(1/2,.87);
\draw[dashed] (0,0)--(0,.87);
\draw[dashed] (-1/2,.87)--(1/4,.43);
\draw[dashed] (1/2,.87)--(-1/4,.43);
\draw[dashed] (-1/2,.87)--(1/2,.87);
\filldraw[blue] (0,0) circle (1pt);
\filldraw[blue] (-1/2,.87) circle (1pt);
\filldraw[blue] (1/2,.87) circle (1pt);
\filldraw[blue] (0,.58) circle (1pt);
\filldraw[red] (-1/4,.43) circle (1pt);
\filldraw[red] (1/4,.43) circle (1pt);
\filldraw[red] (0,.87) circle (1pt);

\end{tikzpicture}
\end{center}
\caption{Four blue and three red points with no monochromatic blue line.}
\label{fig:b2(3)}
\end{figure}

In general, let $(R,B)$ be a $2$-colouring of a simple rank-$r$ real-representable matroid $M$ satisfying the properties (C1), (C2) and (C3) with $|R|=t \geq 3$ and $|B|=b_{r-1}(t)$. Iteratively $2$-summing an $m$-element line onto each red point, we obtain a rank-$(r+t)$ matroid $M'$ where each hyperplane must contain one of the $m$-point lines.

In Section 2, , we prove bounds on the numbers $b_{k}(t)$: we show that $b_{k}(2)=0$ and that $b_{k}(3) \geq {{k+2}\choose{2}}-2$ for all $k$. We also show that $b_{k}(t)$ is linear in $t$ in general.

\begin{theorem}\label{bk_bound}
For $k,t \geq 1$, $b_{k}(t) = \Theta_{k}(t)$.
\end{theorem}

It is easily verified that the Weak Dirac Theorem (see Section 2) holds for any simple matroid in a $\gamma$-Beck class with a constant depending on $\gamma$. It follows that $b_{2}(t)$ is finite not only for real-representable matroids (as discussed above), but more generally for any $\gamma$-Beck class.
Our proof in Section 2 that $b_{k}(t) \leq O_{k}(t)$ is by induction on $k$ and, aside from the invocation of Beck's Theorem, is purely matroidal. Thus $b_{k}(t) = \Theta_{k, \gamma}(t)$ for any $\gamma$-Beck class.

\subsection{Counting hyperplanes}

The proofs of the average bounds in \cite{CGK} rely fundamentally on good approximations of the number of flats of a given rank. As mentioned above, the classical theorem of this type is Beck's approximation \cite{Beck} on the number of lines in a real-representable matroid. Do \cite{Do} and Lund \cite{Lund} independently formulated and proved high-dimensional generalizations of this theorem, which were used in the proofs of the theorems of \cite{CGK}. While Beck's approximation is in terms of a largest line in the matroid, the idea in higher dimensions is to consider largest degenerate sets. For instance, Do's Theorem (see Section 4) roughly says that either the matroid has $\Omega(n^{k})$ rank-$k$ flats, or it contains a very large $k$-degenerate set. 

For our purposes, we keep track of a largest $i$-degenerate set for every possible $i$. For $k \geq 2$ and a matroid $M$, a $k$-{\it{stratification}} of $M$ is a sequence $(X_{2}, \dots, X_{k})$ of subsets of its ground set such that, for each $i \in \{2, \dots, k\}$, $X_{i}$ is $i$-degenerate and $X_{i-1} \subseteq X_{i}$. (In this context, $X_{1}$ denotes the empty set, and $X_{k+1}$ denotes the ground set of $M$.) If, for each $i \in \{2, \dots, k\}$, $X_{i}$ is a largest $i$-degenerate subset of $X_{i+1}$, then the $k$-stratification is called {\it{optimal}}. 

In \cite{Lund}, Lund proved that, if $(X_{2}, \dots, X_{k})$ is an optimal $k$-stratification of a simple real-representable matroid $M$ on $n$ elements, then $M$ has at least $\Omega(n \prod_{i=2}^{k}(n - |X_{i}|))$ rank-$k$ flats, provided that $n-|X_{k}|$ is at least some constant depending only on $k$. 
In \cite{CGK}, this lower bound was combined with the following in order to bound the average hyperplane-size.

\begin{theorem}[\cite{CGK}]\label{aggregate}
For each integer $k \geq 2$, if $M$ is a simple rank-$(k+1)$ matroid on $n$ elements which is not $k$-degenerate, and $(X_{2}, \dots, X_{k})$ is an optimal $k$-stratification of $M$, then
\begin{align*}
    \sum_{H \in \cH(M)} (|H|-k) \leq 2^{k(k-1)} n \prod_{i=2}^{k} (n-|X_{i}|),
\end{align*}
where $\cH(M)$ denotes the collection of hyperplanes of $M$.
\end{theorem}

Because of Theorem \ref{aggregate}, our proof of the Main Theorem entirely consists of obtaining a good approximation on the number of hyperplanes under the assumption that outcome (2) does not hold. An additional consequence of our proof is that the constant in Lund's aforementioned theorem is bounded by $b_{k}(k-2)$:

\begin{theorem}
For $k \geq 2$, there exists $\rho_{k} > 0$ such that, if $M$ is a simple real-representable matroid on $n$ elements, $(X_{2}, \dots, X_{k})$ is an optimal $k$-stratification of $M$, and $n - |X_{k}| > b_{k}(k-2)$, then the number of rank-$k$ flats in $M$ is at least $\rho_{k} n \prod_{i=2}^{k} (n - |X_{i}|)$.
\end{theorem}

\subsection{Matroid terminology}
All of the results in this paper intrinsically concern a finite set of points in space and its spanned hyperplanes. The language of matroid theory allows us to handle such geometric information in a purely combinatorial manner and without referring to any particular coordinatization. Informally, given a finite set $P$ of points in $\mathbb{R}^{d}$ not contained in an affine hyperplane, the corresponding matroid $M$ is a combinatorial object which records which subsets of $P$ are collinear, which subsets are coplanar, and so on. This data is specified by the collection of  subsets of $P$ which are affinely independent; we call these the {\it{independent sets}} of $M$. A maximal independent set of $M$ is called a {\it{basis}}, and a minimal dependent set is called a {\it{circuit}}. A {\it{hyperplane}} of $M$ is a maximal subset of $P$ contained in a $(d-1)$-dimensional affine subspace of $\mathbb{R}^{d}$. For a subset $S \subseteq P$, the {\it{deletion}} of $S$ in $M$ is denoted by $M \del S$ and is defined in the obvious way; the {\it{restriction}} to $S$ is denoted by $M|S$. The {\it{contraction}} of $S$ in $M$, denoted $M \con S$, is the matroid obtained by centrally projecting from the affine subspace spanned by $S$ in $\mathbb{R}^{d}$. We refer the reader to Appendix A of \cite{CGK} for a brief introduction to matroid theory where these notions are defined formally.

\section{Point sets without monochromatic blue hyperplanes}\label{sec:gmep}

The main result of this section is that $b_{k}(t) = \Theta_{k}(t)$. 
The construction for the lower bound shows that $b_{k}(t) \geq \Omega(kt)$, but this does not seem to be best-possible. We now give a construction, generalizing that shown in \cref{fig:b2(3)}, where the number of blue points has quadratic dependence on $k$ for $t=3$.

\begin{proposition}\label{const:improvedGraphical}
    For any $k \geq 2$, $b_k(3) \geq \binom{k+2}{2} - 2$.
\end{proposition}
\begin{proof}
    We first describe the construction, and then show that it has the required properties.
    Our construction is a colouring of the union of the graphic matroid for $K_{k+2}$ with one additional point.
    The rank-$4$ construction is illustrated in \cref{fig:rank4construction}.
    
    Start with the standard real representation of the graphic matroid for $K_{k+2}$.
    In more detail, the image of an edge $\{i,j\} \in E(K_{k+2})$ with $i < j \leq k+1$ is the vector $(x_1, x_2, \ldots, x_{k+1})$ with $x_i = 1$, $x_j = -1$, and all other coordinates equal to $0$.
    For $i < j = k+2$, the image of $\{i,j\}$ is the vector with $x_i = 1$ and all other coordinates equal to $0$.
    
    With a slight abuse of notation, we use $r_1$ both for the edge $r_1 = \{1,2\} \in E(K_{k+2})$, and the point $r_1 = (1,-1,0,0,\ldots,0) \in \mathbb{R}^{k+1}$.
    Similarly, let $r_2 = \{3,4\} \in E(K_{k+2})$ and $r_2 = (0,0,1,-1,0,\ldots,0) \in \mathbb{R}^{k+1}$.
    Let $r_3 = (1,1,-1,-1,0,\ldots,0)$; note that there is no corresponding edge in $K_{k+2}$.
    Let $R = \{r_1,r_2,r_3\}$.
    Let $B = E(K_{k+2}) \setminus R$, and the corresponding points in $\mathbb{R}^{k+1}$.
    Let $M_g$ be the graphic matroid over $K_{k+2}$, and let $M$ be the matroid represented by the points $B \cup R$.

        \begin{figure}[h!]
        \centering
        \begin{minipage}{.3\textwidth}
            \centering
            \begin{tikzpicture}[scale=1.3]
                \draw[red,very thick] (0,1.54)--(.81,.95);
                \draw[red,very thick] (.5,0)--(-.5,0);
                \draw[blue,very thin] (.5,0)--(.81,.95);
                \draw[blue,very thin] (.5,0)--(0,1.54);
                \draw[blue,very thin] (.5,0)--(-.81,.95);
                \draw[blue,very thin] (-.5,0)--(.81,.95);
                \draw[blue,very thin] (-.5,0)--(0,1.54);
                \draw[blue,very thin] (-.5,0)--(-.81,.95);
                \draw[blue,very thin] (-.81,.95)--(.81,.95);
                \draw[blue,very thin] (-.81,.95)--(0,1.54);
                \filldraw[black] (.5,0) circle (2pt) node[anchor=north west] {3};
                \filldraw[black] (.81,0.95) circle (2pt) node[anchor=south west] {2};
                \filldraw[black] (0,1.54) circle (2pt) node[anchor=south] {1};
                \filldraw[black] (-.81, 0.95) circle (2pt) node[anchor = south east] {5};
                \filldraw[black] (-.5, 0) circle (2pt) node[anchor=north east] {4};
            \end{tikzpicture}
        \end{minipage}%
        \begin{minipage}{.7\textwidth}
            $\left[\begin{array}{@{}>{\color{blue}}c>{\color{blue}}c>{\color{blue}}c>{\color{blue}}c>{\color{blue}}c>{\color{blue}}c>{\color{blue}}c>{\color{blue}}c|>{\color{red}}c>{\color{red}}c>{\color{red}}c@{}}
            1 & 0 & 0 & 0 & 1 & 1 & 0 & 0 & 1 & 0 & 1  \\
            0 & 1 & 0 & 0 & 0 & 0 & 1 & 1 & -1 & 0 & 1 \\
            0 & 0 & 1 & 0 & 0 & -1 & -1 & 0 & 0 & 1 & -1 \\
            0 & 0 & 0 & 1 & -1 & 0 & 0 & -1 & 0 & -1 & -1
            \end{array}\right]$
        \end{minipage}
        \caption{On the left is a vertex-labeled $K_{5}$ with two red edges. On the right is its rank $4$ representation over $\mathbb{R}$, with one extra red point added. There is no monochromatic blue plane in the representation.}
        \label{fig:rank4construction}
    \end{figure}
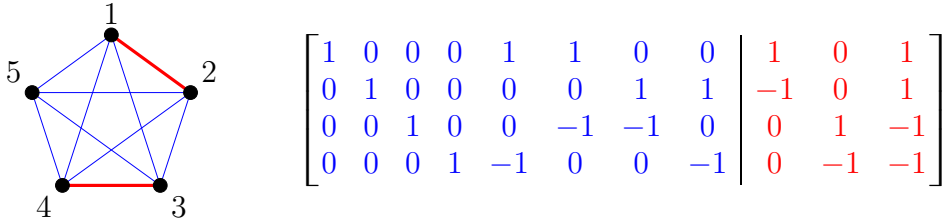

    Clearly $(R,B)$ satisfy (C1) for $M$. To verify (C2) and (C3), we use the fact that, for $\ell \geq 0$, the rank-$(k+1-\ell)$ flats of $M_{g}$ are in one-to-one correspondence with the partitions of $[k+2]$ into $\ell+1$ nonempty parts, and the elements of such a flat are the edges whose vertices belong to the same set in the partition.

    Let us first show that (C2) is satisfied; see \cref{fig:cutPlane} for an illustration.
    Let $F$ be a hyperplane of $M_{g}$ and let $(V_{1},V_{2})$ be the partition of $[k+2]$ corresponding to $F$ with $k+2 \in V_{2}$.
    The equation of $F$ is $\sum_{i \in V_1} x_{i} = 0$.
    If $F$ contains neither $r_1$ nor $r_2$, then $V_1 \cap \{1,2\} = 1$ and $v_1 \cap \{3,4\} = 1$.
    In this case, it is clear from the equation that $F$ contains $r_3$.

    \begin{figure}[h!]
            \centering
            \begin{tikzpicture}[scale=1.3]
                \draw[red,very thick, dashed] (0,1.54)--(.81,.95);
                \draw[red,very thick, dashed] (.5,0)--(-.5,0);
                \draw[blue,very thin] (.5,0)--(.81,.95);
                \draw[blue,very thin, dashed] (.5,0)--(0,1.54);
                \draw[blue,very thin, dashed] (.5,0)--(-.81,.95);
                \draw[blue,very thin, dashed] (-.5,0)--(.81,.95);
                \draw[blue,very thin] (-.5,0)--(0,1.54);
                \draw[blue,very thin] (-.5,0)--(-.81,.95);
                \draw[blue,very thin, dashed] (-.81,.95)--(.81,.95);
                \draw[blue,very thin] (-.81,.95)--(0,1.54);
                \filldraw[black] (.5,0) circle (2pt) node[anchor=north west] {3};
                \filldraw[black] (.81,0.95) circle (2pt) node[anchor=south west] {2};
                \filldraw[black] (0,1.54) circle (2pt) node[anchor=south] {1};
                \filldraw[black] (-.81, 0.95) circle (2pt) node[anchor = south east] {5};
                \filldraw[black] (-.5, 0) circle (2pt) node[anchor=north east] {4};
                \draw plot[smooth cycle] coordinates {(0.2,0) (.6,1.25) (.81,1.5) (1.3,.95) (1,0) (.8,-.4) (.4,-.4)};
                \draw plot[smooth cycle] coordinates {(-0.2,0) (.2,1.25) (.2, 1.8) (0, 2) (-1.2, 1.2) (-1.2, .8) (-.8, -.4) (-.4, -.4)};
                \filldraw[black] (1.28,1) circle (0pt) node[anchor=west] {$V_1$};
                \filldraw[black] (-1.2, 1.2) circle (0pt) node[anchor=east] {$V_2$};
                \filldraw[black] (1.2,.2) circle (0pt) node[anchor=west] {$F = \{(x_1,x_2,x_3,x_4): x_2 + x_3 = 0\}$};
            \end{tikzpicture}
        \caption{While the plane corresponding to this partition of $K_5$ doesn't contain $r_1$ or $r_2$, it does contain the point $r_3 = (1,1,-1,-1)$.}
        \label{fig:cutPlane}
    \end{figure}
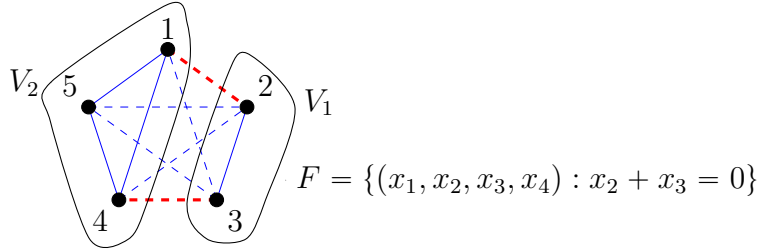
    
    Now we show that (C3) is satisfied; see \cref{fig:codim2}.
    Let $L$ be a line spanned by blue points of $M_{g}$. Then, in $K_{k+2}$, $L$ is either a disjoint pair of blue edges or a triangle with at least two blue edges. Let $i \in \{1,2\}$ and $j \in \{3,4\}$ such that $\{i,j\} \notin L$. Now choose a partition $(V_{1}, V_{2}, V_{3})$ of $[k+2]$ such that $\{i,j\} \subseteq V_{1}$ and $V_{2}$ and $V_{3}$ each contain one of the elements of $\{1,2,3,4\} \setminus \{i,j\}$, and such that no edge of $L$ is contained in $V_{1}, V_{2}$ or $V_{3}$. Let $F$ be the rank-$(k-1)$ flat corresponding to $(V_{1}, V_{2}, V_{3})$. Then $L$ is skew to $F$ by construction, and $r_{1}, r_{2} \notin F$. Furthermore, because the points of $F$ satisfy $\sum_{i \in V_{1}} x_{i} = 0$ and $\sum_{j \in V_{2}} x_{j}=0$, we have that $r_{3} \notin F$. 
    \end{proof}

    \begin{figure}[h!]
            \centering
            \begin{tikzpicture}[scale=1.3]
                \draw[red,very thick] (0,1.54)--(.81,.95);
                \draw[red,very thick] (.5,0)--(-.5,0);
                \draw[blue,very thin, dashed] (.5,0)--(.81,.95);
                \draw[blue,very thin, dashed] (.5,0)--(0,1.54);
                \draw[blue,very thin,dashed] (.5,0)--(-.81,.95);
                \draw[blue,very thick] (-.5,0)--(.81,.95);
                \draw[blue,very thin,dashed] (-.5,0)--(0,1.54);
                \draw[blue,very thin,dashed] (-.5,0)--(-.81,.95);
                \draw[blue,very thin,dashed] (-.81,.95)--(.81,.95);
                \draw[blue,very thick] (-.81,.95)--(0,1.54);
                \filldraw[black] (.5,0) circle (2pt) node[anchor=north west] {3};
                \filldraw[black] (.81,0.95) circle (2pt) node[anchor=south west] {2};
                \filldraw[black] (0,1.54) circle (2pt) node[anchor=south] {1};
                \filldraw[black] (-.81, 0.95) circle (2pt) node[anchor = south east] {5};
                \filldraw[black] (-.5, 0) circle (2pt) node[anchor=north east] {4};
                \draw plot[smooth cycle] coordinates {(0.2,0) (.6,1.25) (.81,1.5) (1.3,.95) (1,0) (.8,-.4) (.4,-.4)};
                \draw plot[smooth cycle] coordinates {(-0.2,0) (-.8,1.5) (-1.2, 1.2) (-1.2, .8) (-.8, -.4) (-.4, -.4)};
                \draw plot[smooth cycle] coordinates {(0.2,1.3) (0.2, 2) (-0.2, 2) (-0.2, 1.3)};
                \filldraw[black] (1.28,1) circle (0pt) node[anchor=west] {$V_1$};
                \filldraw[black] (0.2,1.8) circle (0pt) node[anchor=west] {$V_2$};
                \filldraw[black] (-1.2, 1.2) circle (0pt) node[anchor=east] {$V_3$};
                \filldraw[black] (1.2,.2) circle (0pt) node[anchor=west] {$F = \{(x_1,x_2,x_3,x_4): x_2 + x_3 = 0 \text{ and } x_1 = 0\}$};
            \end{tikzpicture}
        \caption{$L$ is the blue line spanned by $(1,0,0,0)$ and $(0,1,0,-1)$. The co-dimension $2$ flat $F$ is the line passing through $(0,0,0,1)$ and $(0,1,-1,0)$. $F$ is skew to $L$, and doesn't contain $r_1$, $r_2$, or $r_3$.}
        \label{fig:codim2}
    \end{figure}
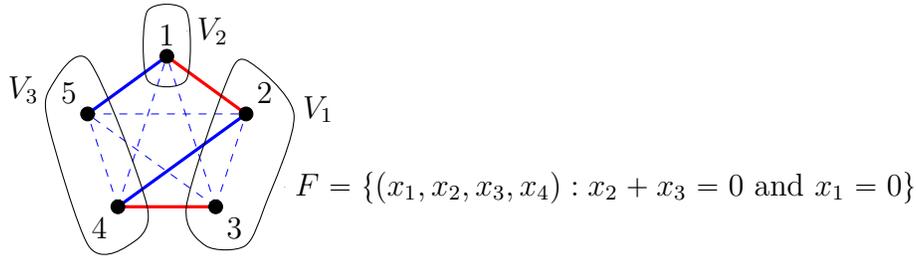

For the general lower bound on $b_{k}(t)$, we construct an arrangement where the blue points comprise several parallel copies of a planar set of points determining few directions. The red points are in the hyperplane at infinity, corresponding to the directions of lines spanned in each of the planar point sets.

\begin{proposition}\label{lower_bk}
    Let $k \geq 3$ and $t \geq 3$.
    Then, $b_k(t) \geq (k-1)t$.
    Furthermore, if $t$ is even, then $b_k(t) \geq (k-1)(t+1)$.
\end{proposition}
\begin{proof}
    Let $S$ be a set of points in the plane that determines at most $t$ directions, such that $|S|=t+1$ if $t$ is even and $|S|=t$ if $t$ is odd.
    For even $t$, there are many configurations of $t+1$ points that determine exactly $t$ directions; 
    Jamison and Hill \cite{jamison1983catalogue} described four infinite families and 102 sporadic examples that are not known to be members of any infinite families.
    For odd $t$, we simply use a construction for $t-1$.
    
    We use projective coordinates to describe the blue points $B$ and red points $R$.
    Roughly speaking, $B$ is the union of $k-1$ parallel copies of $S$, and $R$ is the set points at infinity corresponding to the directions determined by $S$. Specifically, let $B = \bigcup_{i \in [k-1]}B_i$, where
    \[B_i = \{(1,a,b,x_1,\ldots,x_{k-1}): (1,a,b) \in S, x_i=1, x_j =0 \text{ for } j \neq i\}, \]
    and let
    \[R=\{(0,a-a',b-b',0,\ldots,0):(1,a,b), (1,a',b') \in S, (a,b) \neq (a',b')\}.\]
    Now $|R|=t$ and $|B| = (k-1)|S|$. Observe that, for any $x,y \in B_{i}$, the point $x-y$ is in the line spanned by $x$ and $y$, and $x-y \in R$.

    Now we show that $(R,B)$ satisfies (C1), (C2) and (C3) (for the matroid $R \cup B$ represents). (C1) is immediate. For (C2), let $H$ be a rank-$k$ flat spanned by an independent set $I \subseteq B$. By the pigeonhole principle, $|I \cap B_{i}| \geq 2$ for some $i \in [k-1]$, and so $H$ contains a red point. For (C3), let $L$ be a line spanned by points of $B$, and let $L'$ be the projection of $L$ onto the first three coordinates. Since the points of $S$ are not all collinear, there is a point $(1,a,b) \in S \setminus L'$.
    Let $I$ be the set of $k-1$ points in $B$ with first three coordinates $(1,a,b)$, and let $H$ be the rank-$(k-1)$ flat spanned by $I$. Then $H$ is skew to $L$ by construction, and it does not contain any red points because every point in $H$ is non-zero in one of its last $k-1$ coordinates.
\end{proof}

It is easy to see that $b_{2}(2) = 0$, since at least three red points are needed to hit the three lines spanned by three non-collinear blue points. Now we show that $b_{k}(2) = 0$ for all $k$. The proof of this fact does not depend in any way on the representability of $M$.  

\begin{proposition}
For $k \geq 2$, let $M$ be a simple rank-$(k+1)$ matroid. If $(R,B)$ is a $2$-colouring of $M$ satisfying (C1), (C2) and (C3), then $|R| \geq 3$.
\end{proposition}

\begin{proof}
Suppose to the contrary that $R = \{r_{1}, r_{2}\}$. Assume first that $r_{1}$ and $r_{2}$ are in different connected components of $M$, and let $(X_{1}, X_{2})$ be a separation of $M$ with $r_{1} \in X_{1}$ and $r_{2} \in X_{2}$. By (C1) and (C3), $r(B) = r(M)$, which means that $|X_{1}|, |X_{2}| \geq 2$. We claim that $|X_{1}|, |X_{2}| \geq 3$. Indeed, suppose that $X_{1} = \{r_{1}, e\}$ for some $e \in B$. Choose $f \in B \cap X_{2}$, and let $F$ be a monochromatic-blue rank-$(k-1)$ flat in $M$ skew to $\{e,f\}$. Since $r(F \cap X_{2}) \leq r(X_{2})-1 = k-2$, we have $r_{1} \in F$, a contradiction. Thus $|X_{1}|, |X_{2}| \geq 3$. Now let $e,f \in X_{1} \cap B$, and let $F$ be a monochromatic-blue rank-$(k-1)$ flat in $M$ skew to $\{e,f\}$. Then $r_{2}$ is in neither of the hyperplanes $\cl(F \cup \{e\})$ and $\cl(F \cup \{f\})$, and $r_{1}$ is in at most one of them. This contradicts (C2).

Now assume that $r_{1}$ and $r_{2}$ are in the same connected component of $M$, and let $C$ be a circuit containing both of them. Let $I$ be a basis of $M|B$ containing $C \del \{r_{1}, r_{2}\}$, and let $C_{1}$ and $C_{2}$ be the fundamental circuits of $I \cup \{r_{1}\}$ and $I \cup \{r_{2}\}$ respectively. By choice of $I$, there exists $e \in C_{1} \cap C_{2}$. But then $r_{1}, r_{2} \notin \cl(I \del \{e\})$, contrary to (C2).
\end{proof}

We finish this section by showing that $b_{k}(t) \leq O_{k}(t)$, thus completing the proof of Theorem \ref{bk_bound}. The proof relies on the matroidal operation of {\it{principal truncation of a line}}. Geometrically, this is the operation of centrally projecting from a generic point on a line. (This operation is defined matroidally and in full generality in Section 3.) 
The proof also relies on two classic theorems from discrete geometry, the first of which is the Weak Dirac Theorem.

\begin{dirac}\cite{Beck, SzemerediTrotter}
There is a number $c_{d} > 0$ such that, if $M$ is a simple real-representable matroid on $n$ elements and $r(M) \geq 3$, then there is an element of $M$ incident to at least $c_{d}n$ lines.
\end{dirac}

The Weak Dirac Theorem follows from Beck's theorem, and hence holds for any $\gamma$-Beck class of matroids, with $c_d$ depending on $\gamma$.

As an historical aside, the question of determining the largest possible value of $c_d$ is classical.
For example, the Strong Dirac Conjecture posits that, if $M$ is a real-representable matroid on $n$ elements with $r(M) \geq 3$, then there is an element of $M$ incident to at least $n/2 - O(1)$ lines.
The best-known result in this direction is that there is always a point incident to  $\lceil n/3 \rceil +1$ lines.
This follows from a result of Langer in algebraic geometry \cite{langer2003logarithmic, han2017note}. 
This result holds for complex-representable matroids and is tight in that setting. 
The situation for orientable matroids is less satisfactory since Langer's theorem does not apply.
It is known that there is an infinite family of rank-$3$ orientable matroids such that each element is incident to at most $(4/9)n$ lines \cite{lund2014pseudoline}, so the oriented matroid analog to the Strong Dirac Conjecture is not true.

The second classic theorem that we use is a high-dimensional extension of Beck's theorem on the number of lines, also proved by Beck in \cite{Beck}.

\begin{beck}\cite{Beck}
For each $k \geq 2$, there exist $\epsilon_{k}, \gamma_{k} \in (0,1)$ such that, if $M$ is a simple rank-$(k+1)$ real-representable matroid on $n$ elements, then either
\begin{enumerate}
    \item there is a hyperplane in $M$ of size at least $\epsilon_{k}n$, or
    \item $M$ has at least $\gamma_{k}n^{k}$ hyperplanes.
\end{enumerate}
\end{beck}

In fact, for any $\epsilon_k < (1/k) \lfloor k/2 \rfloor$, there is a constant $\gamma_k$ such that Beck's Theorem holds for those $\epsilon_k, \gamma_k$ \cite{lund2021two}.

Our proof that $b_{k}(t) \leq O_{k}(t)$ proceeds as follows. We consider a $2$-colouring $(R,B)$ of a rank-$(k+1)$ real-representable matroid $M$ satisfying (C1), (C2) and (C3), and show that $|B| \leq O_{k}(|R|)$ by induction on $k \geq 2$. Using the Weak Dirac Theorem, we start by contracting a blue element incident to $\Omega(|B|)$ lines in $M$. Then we colour the resulting matroid in the natural way, where the red points correspond to lines containing a red element. This colouring satisfies (C1) and (C2), so if it satisfies (C3), then we win by induction. If it does not, then we can principally truncate a line to obtain a matroid which (with the natural resulting colouring) again satisfies (C1) and (C2). Using part (ii) of the following lemma, we can repeat this process until we obtain a bicoloured truncated matroid satisfying (C3). 

\begin{lemma}\label{gmep_helpful}
For $k \geq 2$, let $M$ be a simple rank-$(k+1)$ matroid and let $L$ be a line of $M$. Let $(R,B)$ be a $2$-colouring of $M$ such that every hyperplane of $M$ contains a red element.
\begin{enumerate}
    \item If there is a monochromatic-blue rank-$(k-1)$ flat skew to $L$, then $|L| \leq |R|$.
    \item If every rank-$(k-1)$ flat skew to $L$ contains a red element, then every hyperplane in the principal truncation of $L$ contains a red element.
\end{enumerate}
\end{lemma}

\begin{proof}
For (1), observe that if there is a rank-$(k-1)$ flat $F$ skew to $L$ with $F \cap R = \emptyset$, then for all $e \in L$, $\cl_{M}(F \cup \{e\}) \cap R \neq \emptyset$, and so $|L| \leq |R|$. For (2), observe that the set of hyperplanes of the principal truncation of $L$ consists precisely of the hyperplanes of $M$ containing $L$ and the rank-$(k-1)$ flats of $M$ skew to $L$.
\end{proof}

There is only one problem with the argument described above. Namely, in order to derive the desired bound on $|B|$ after applying induction, we need the number of blue points in the truncated matroid to be $\Omega(|B|)$. We ensure this by initially bounding the size of each flat in $M$ by $O(|B|)$. This will bound the number of blue points in each truncated line as needed.

\begin{proof}[Proof of Theorem \ref{bk_bound}.]
The lower bound follows by Proposition \ref{lower_bk}. For the upper bound, we proceed by induction on $k \geq 1$. By definition, $b_{1}(t) = 0$ for all $t \geq 1$. Now let $k \geq 2$ and assume that, for all $i \in \{1, \dots, k-1\}$, there is a constant $C_{i} \geq 0$ such that $b_{i}(t) \leq C_{i}t$ for all $t \geq 1$. Let $M$ be a simple rank-$(k+1)$ real-representable matroid, and let $(R,B)$ be a $2$-colouring of $M$ satisfying (C1), (C2) and (C3). Let $t = |R|$. 

Using the constants $c_{d}$ from the Weak Dirac Theorem and $\epsilon_{i}$ and $\gamma_{i}$ from Beck's Theorem for Hyperplanes, we define $\delta_{k} = (1/4k)c_{d}$, and for each $i \in \{1, \dots, k-1\}$, we recursively define $\delta_{i-1} = (1/2)\epsilon_{i-1}\delta_{i}$. We may assume that $|B| \geq \max \{ \frac{2t}{c_{d}\gamma_{i-1}\delta_{i}^{i-1}} : 2 \leq i \leq k \}$, where we set $\gamma_{1}=1$. 

\begin{claim}\label{flat_bound}
For each $i \in \{2, \dots, k\}$, every rank-$i$ flat of $M|B$ has size at most $\delta_{i}|B|$. 
\end{claim}

\begin{proof}[Proof of claim.]
Suppose otherwise, and choose $i \in \{2, \dots, k\}$ as small as possible such that there exists a rank-$i$ flat $F$ of $M|B$ with $|F| \geq \delta_{i}|B|$. By Lemma \ref{gmep_helpful}, every line of $M|B$ has size at most $t < \delta_{2}|B|$, and so $i \geq 3$. By choice of $i$, every hyperplane of $M|F$ has size at most $\delta_{i-1}|B| < \epsilon_{i-1}|F|$. Hence, by Beck's Theorem for Hyperplanes, there are at least $\gamma_{i-1}\delta_{i}^{i-1}|B|^{i-1}$ hyperplanes in $M|F$. 

By (C3), there is a monochromatic-blue rank-$(k+1-i)$ flat $P$ in $M$ which is skew to $F$. To see this, take a line $L$ of $M|F$ and a monochromatic-blue rank-$(k-1)$ flat $F'$ skew to $L$. Then let $P = \cl_{M}(I)$, where $I$ is a maximal independent subset of $F' \del F$ in $M \con F$. Now $(M \con P)|F = M | F$, and every hyperplane of $M \con P$ contains a red element. Since each red element is spanned by at most $|B|^{i-2}$ hyperplanes of $M|F$ in $M \con P$, we have that there are at least
\begin{align*}
    \gamma_{i-1}\delta_{i}^{i-1}|B|^{i-1} - t|B|^{i-2} = \Big( \gamma_{i-1}\delta_{i}^{i-1} - \frac{t}{|B|} \Big)|B|^{i-1} > 0
\end{align*}
hyperplanes in $M \con P$ containing no red element, a contradiction.
\end{proof}

By the Weak Dirac Theorem, there exists $e \in B$ incident to at least $c_{d}|B|$ lines in $M|B$. Let $M'$ be a simplification of $M \con e$, and let $(R',B')$ be a $2$-colouring of $M'$ where $f \in R'$ if and only if $\cl_{M}(\{e,f\}) \cap R \neq \emptyset$. Then $(R',B')$ satisfies (C1) and (C2). Also $|R'| \leq t$ and $|B'| \geq c_{d}|B| - t \geq \frac{1}{2}c_{d}|B|$. Henceforth, we refer to the elements of $R'$ as {\it{red}} and the elements of $B'$ as {\it{blue}}. 

We have the following as an immediate consequence of Claim \ref{flat_bound}.

\begin{claim}\label{con_flat_bound}
For $1 \leq i \leq k-1$, if $F$ is a rank-$i$ flat of $M'|B'$, then $|F| \leq \frac{2\delta_{i}}{c_{d}}|B'|$. 
\end{claim}

By repeated application of Lemma \ref{gmep_helpful}, there is a matroid $M''$, obtained from $M'$ by principally truncating $\ell \leq k-2$ lines in turn, such that every hyperplane of $M''$ contains a red element, and, for every line of $M''$ containing two monochromatic-blue points, there is a monochromatic-blue rank-$(r(M'')-2)$ flat of $M''$ skew to it. Let $B''$ be the set of all blue elements which are not contained in a non-trivial parallel class of $M''$. (Equivalently, $B''$ is the set of blue elements not contained in one of the lines which were truncated to obtain $M''$.) For each line which was principally truncated to obtain $M''$, its set of blue elements is a flat of $M'|B'$ of rank at most $k-1$. Therefore, by Claim \ref{con_flat_bound},
\begin{align*}
    |B''| \geq |B'| - \frac{2 \ell \delta_{k}}{c_{d}}|B'| \geq \frac{1}{2} |B'| \geq \frac{c_{d}}{4}|B|.
\end{align*}
By induction, $|B''| \leq C_{k - \ell}t$. The result follows.
\end{proof}

We remark that the preceding theorem leads to a bound of $b_k(t) \leq c^k t$ for some constant $c$.
Our general lower bound is $b_k(t) \geq (k-1)t$, so there remains a substantial gap, especially for large $k$.

\section{Principal truncation and stratification}

In the proof of the Main Theorem, we invoke the inductive hypothesis by reducing to a smaller matroid using the operation of principal truncation. Geometrically, this corresponds to placing a point in general position inside a specified flat, and then centrally projecting at this  point. Let $M$ be a matroid and let $F$ be a set in $M$ of rank at least two. The {\it{principal truncation of $F$ in $M$}} is the matroid with $N$ ground set $E(N) = E(M)$ and rank function
\begin{align*}
    r_{N}(X) = 
    \begin{cases}
        r_{M}(X) - 1 \ \ {\mathrm{if}} \ F \subseteq \cl_{M}(X) \\
        r_{M}(X) \ \ \ \ \ \ \ {\mathrm{otherwise}}.
    \end{cases}
\end{align*}
Equivalently, $N$ is the matroid obtained from $M$ by freely extending in $\cl_{M}(F)$ by a new element and then contracting this element. (This means that a $\gamma$-Beck class is closed under principal truncation.) Note that if $M$ is a simple matroid, then so is $N$, unless $r_{M}(F)=2$, in which case $\cl_{M}(F)$ is a point in $N$, and $N \del \cl_{M}(F)$ is simple. For any $1 \leq k \leq r_{M}(F) - 1$, this operation may be reiterated $k$ times; we refer to this as the {\it{$k$-fold principal truncation}} of $F$ in $M$. The $(r_{M}(F)-1)$-fold principal truncation of $F$ in $M$ is called the {\it{complete principal truncation}} of $F$ in $M$, which we denote by $M \div F$. 

We begin by laying out some standard properties of principal truncation which will be particularly helpful to us. We will prove these properties in full, although they are well-known consequences of the fact that principal truncation is a quotient operation on a matroid, see \cite[Section 7.3]{oxley2006matroid}. 

\begin{proposition}\label{helpful}
Let $M$ be a matroid, and let $F \subseteq E(M)$ with $2 \leq r_{M}(F) \leq r(M)-2$. Let $N$ be the $1$-fold principal truncation of $F$ in $M$.
\begin{enumerate}
    \item For $X \subseteq E(M)$, if $F \nsubseteq \cl_{N}(X)$, then $F \nsubseteq \cl_{M}(X)$. Consequently, if $F \nsubseteq \cl_{M \div F}(X)$, then $F \nsubseteq \cl_{M}(X)$.
    \item If $H$ is a flat in $N$, then $H$ is a flat in $M$. Consequently, if $H$ is a flat in $M \div F$, then $H$ is a flat in $M$.
    \item If $H$ is a flat in $M \div F$ which does not contain $F$, then $H$ is a rank-$r_{M \div F}(H)$ flat in $M$ which is skew to $F$.
    \item If $H$ is a hyperplane in $M$ which does not contain $F$, then there exists $X \subseteq H$ such that $X$ is a hyperplane in $M \div F$ (which is disjoint from $F$).
\end{enumerate}
\end{proposition}

\begin{proof}
For (1), suppose that $F \nsubseteq \cl_{N}(X)$. Then
\begin{align*}
    r_{M}(X \cup F) - 1 = r_{N}(X \cup F) \geq r_{N}(X) + 1 \geq r_{M}(X), 
\end{align*}
as desired. 

For (2), suppose to the contrary that there exists $e \in \cl_{M}(H) \del H$. If $F \subseteq \cl_{M}(H)$, then
\begin{align*}
    r_{N}(H \cup \{e\}) = r_{M}(H \cup \{e\}) - 1 = r_{M}(H)-1 = r_{N}(H),
\end{align*}
a contradiction. On the other hand, if $F \nsubseteq \cl_{M}(H)$, then
\begin{align*}
    r_{N}(H \cup \{e\}) - 1 = r_{N}(H) = r_{M}(H) = r_{M}(H \cup \{e\}),
\end{align*}
again a contradiction.

For (3), it suffices to prove that if $H$ is a flat in $N$ which is skew to $F$, then $H$ is a flat in $M$ which is skew to $F$. (The result then follows immediately by induction on $r_{M}(F)$.) Indeed, let $H$ be a flat in $N$ which is skew to $F$. By (2), $H$ is a flat in $M$. By (1), $F \nsubseteq \cl_{M}(H)$, and so $r_{M}(H) = r_{N}(H)$. Hence
\begin{align*}
    r_{M}(X \cup F) - 1 = r_{N}(X \cup F) = r_{N}(X) + r_{N}(F) = r_{M}(X) + r_{M}(F) - 1,
\end{align*}
and so $H$ and $F$ are skew in $M$.

For (4), take a flat $X \subseteq H$ in $M$ which is skew to $F$ and has rank $r(M)-r_{M}(F)$.
\end{proof}

The reason principal truncation is the right operation to use in our proof is that, subject to mild assumptions, it behaves very nicely with respect to an optimal stratification. The main lemma of this section will allow us to obtain an optimal stratification in our truncated matroid that preserves much of the relevant data from the original. We first prove a slightly weaker result. For an optimal $k$-stratification $(X_{2}, \dots, X_{k})$ of a matroid $M$ and $i \in \{2, \dots, k\}$, we call $(X_{i}, \dots, X_{k})$ an optimal {\it{partial}} $k$-stratification of $M$.

\begin{proposition}\label{truncation_prop}
For $k \geq 4$, let $M$ be a simple rank-$(k+1)$ matroid, and let $(X_{2}, \dots, X_{k})$ be an optimal $k$-stratification of $M$. For $i \in \{3, \dots, k-1\}$, if $X_{i}$ is a rank-$i$ flat in $M$, and $|X_{i-1}| < \frac{1}{k}|X_{i}|$, then $(X_{i}, \dots, X_{k})$ is an optimal partial $(k-1)$-stratification of the $1$-fold principal truncation of $X_{i}$ in $M$.
\end{proposition}

\begin{proof}
Let $N$ denote the $1$-fold principal truncation of $X_{i}$ in $M$. Clearly $X_{i}$ is a rank-$(i-1)$ flat in $N$, and so it is $(i-1)$-degenerate in $N$. Now let $\ell \in \{i+1, \dots, k\}$, and let $(F_{1}, \dots, F_{t})$ be a collection of flats of $M$, each having rank at least two, such that $X_{\ell} = F_{1} \cup \dots \cup F_{t}$ and $\sum_{j=1}^{t}(r_{M}(F_{j}) -1 ) \leq \ell-1$. Choose $j \in \{1, \dots, t\}$ such that $|X_{i} \cap F_{j}|$ is as large as possible. Since $t \leq k$ and $X_{i} \subseteq X_{\ell}$, $|X_{i} \cap F_{j}| \geq \frac{1}{k}|X_{i}|$. If $X_{i} \nsubseteq F_{j}$, then $X_{i} \cap F_{j}$ is an $(i-1)$-degenerate set, contrary to the choice of $X_{i-1}$. Thus $X_{i} \subseteq F_{j}$. It follows that $X_{\ell}$ is $(\ell - 1)$-degenerate in $N$.

It remains to show that, for each $\ell \in \{i, \dots, k\}$, $X_{\ell}$ is a largest $(\ell-1)$-degenerate subset of $X_{\ell+1}$ in $N$. (Here $X_{k+1} = E(N)$.) Indeed, for $\ell \in \{i, \dots, k\}$, suppose to the contrary that there exists $Y \subseteq X_{\ell+1}$ such that $|Y| > |X_{\ell}|$ and $Y$ is $(\ell-1)$-degenerate in $N$. Let $(F_{1}, \dots, F_{t})$ be a collection of flats of $N$, each with rank at least two, such that $Y = F_{1} \cup \dots \cup F_{t}$ and $\sum_{j=1}^{t} (r_{N}(F_{j})-1) \leq \ell - 2$. Recall that we may assume that the flats $F_{1}, \dots, F_{t}$ are pairwise-skew in $N$, and so at most one of them contains $X_{i}$. If $X_{i} \subseteq F_{j}$ for some $j \in \{1, \dots, t\}$, then, by Proposition \ref{helpful}(1), $Y$ is $\ell$-degenerate in $M$; otherwise $Y$ is $(\ell-1)$-degenerate in $M$, in either case a contradiction to the choice of $X_{\ell}$.
\end{proof}

Now we are ready to prove our main lemma on principal truncation.

\begin{lemma}\label{truncation_lemma}
For $k \geq 2$, let $M$ be a simple rank-$(k+1)$ matroid, and let $(X_{2}, \dots, X_{k})$ be an optimal $k$-stratification of $M$. For $i \in \{2, \dots, k-1\}$, if $X_{i}$ is a rank-$i$ flat in $M$, and $|X_{i-1}| +1 < \frac{1}{k}|X_{i}|$, then $(X_{i+1} \setminus X_{i}, \dots, X_{k} \setminus X_{i})$ is an optimal $(k+1-i)$-stratification of $(M \div X_{i}) \setminus X_{i}$.
\end{lemma}

\begin{proof}
We first prove the result for $i=2$, and then proceed by induction. Let $\ell \in \{3, \dots, k\}$, and let $(F_{1}, \dots, F_{t})$ be a collection of flats of $M$, each having rank at least two, such that $X_{\ell} = F_{1} \cup \dots \cup F_{t}$ and $\sum_{j=1}^{t} (r_{M}(F_{j}) - 1) \leq \ell -1$. If $X_{2} \nsubseteq F_{j}$ for all $j \in \{1, \dots, t\}$, then $X_{2} \leq t < k$, a contradiction. It follows that, for all $\ell \in \{3, \dots, k\}$, $X_{\ell}$ is an $(\ell-1)$-degenerate set in $M \div X_{2}$. Thus $(X_{3} \del X_{2}, \dots, X_{k} \del X_{2})$ is a $(k-1)$-stratification of $(M \div X_{2}) \del X_{2}$; it remains to show that it is optimal. Indeed, let $\ell \in \{3, \dots, k\}$, and suppose to the contrary that there exists $Y \subseteq X_{\ell+1} \del X_{2}$ such that $|Y| > |X_{\ell} \del X_{2}|$ and $Y$ is $(\ell-1)$-degenerate in $M \div X_{2}$. Then, by Proposition \ref{helpful}(1), $Y \cup X_{2}$ is $\ell$-degenerate in $M$, a contradiction. This completes the proof for $i = 2$.

Now suppose that $i \in \{3, \dots, k-1\}$, and let $N$ be the $1$-fold principal truncation of $X_{i}$ in $M$. Then $X_{i}$ is a rank-$(i-1)$ flat in $N$, and, by Proposition \ref{truncation_prop}, $(X_{i}, \dots, X_{k})$ is an optimal partial $(k-1)$-stratification of $N$. Let $Z$ be a largest $(i-2)$-degenerate set in $N$ contained in $X_{i}$. Since $X_{i}$ is a rank-$i$ flat in $M$ and $Z \subseteq X_{i}$, $Z$ is an $(i-2)$-degenerate set in $M$, and so $|X_{i}| > k(|Z|+1)$. By induction, $(X_{i+1} \del X_{i}, \dots, X_{k} \del X_{i})$ is an optimal $(k+1-i)$-stratification of $(N \div X_{i}) \del X_{i} = (M \div X_{i}) \del X_{i}$.
\end{proof}

\section{Proof of the main theorem}

Along with principal truncation, the other main tool we use in the proof is the following theorem of Do \cite{Do}, which allows us to assume the existence of a huge degenerate set in the matroid under consideration.

\begin{do_thm}\cite{Do}
For every $\epsilon >0$ and integer $k \geq 2$, there exists $\gamma_{k}(\epsilon) > 0$ such that, if $M$ is a real-representable matroid on $n$ points, then either the number of rank-$k$ flats in $M$ is at least $\gamma_{k}(\epsilon)n^{k}$, or $M$ has a $k$-degenerate set of size at least $(1-\epsilon)n$.
\end{do_thm}

To facilitate induction, we prove something slightly stronger and more technical than the Main Theorem. In order to state this result, we formulate a special notion of degeneracy for bicoloured matroids. For a sequence of sets $(F_{1}, \dots, F_{t})$ in a matroid $M$, we let $M \div F_{1} \div \dots \div F_{t}$ denote the matroid obtained by completely principally truncating $F_{1}$ in $M$, then completely principally truncating $F_{2}$ in $M \div F_{1}$, and so on. For $k \geq 2$ and $\ell \geq 0$, let $M$ be a simple rank-$(k+1)$ matroid, and let $(R,B)$ be a $2$-colouring of $M$ with $|R|=\ell$. We say that $(M,R,B)$ is $(k,\ell)$-{\it{degenerate}} if either $B$ is $k$-degenerate, or there exists a sequence $(F_{1}, \dots, F_{t})$ of disjoint subsets of $B$ such that
\begin{enumerate}
    \item[(D1)] $t \leq k-2$;
    \item[(D2)] $2 \leq r_{M}(F_{1}) \leq k-1$, and, for each $j \in \{2, \dots, t\}$, $2 \leq r_{M_{j}}(F_{j}) \leq r(M_{j}) -2$, where $M_{j} = M \div F_{1} \div \dots \div F_{j-1}$; and
    \item[(D3)] every hyperplane in $M \div F_{1} \div \dots \div F_{t}$ has non-empty intersection with $R \cup F_{1} \cup \dots \cup F_{t}$.
\end{enumerate}

Now we can state our technical version of the Main Theorem.

\begin{theorem}\label{tech_main}
For integers $k \geq 2$ and $\ell\geq 0$, there exist $n_{k}(\ell) \geq 0$ and $\rho_{k}(\ell) > 0$ such that the following holds. Let $M$ be a simple rank-$(k+1)$ real-representable matroid, let $(R,B)$ be a $2$-colouring of $M$ where $|R|=\ell$, and let $(X_{2}, \dots, X_{k})$ be an optimal $k$-stratification of $M|B$. If $|B| \geq n_{k}(\ell)$ and $(M,R,B)$ is not $(k,\ell)$-degenerate, then there are at least $\rho_{k}(\ell) |B| \prod_{i=2}^{k}(|B|-|X_{i}|)$ monochromatic-blue hyperplanes in $M$.
\end{theorem}

Before proving Theorem \ref{tech_main}, we show that it indeed implies the Main Theorem. To prove the second part of the Main Theorem, we require the following lemma.

\begin{lemma}\label{gmep_structure}
For $k \geq 2$, let $M$ be a simple rank-$(k+1)$ real-representable matroid, and let $(R,B)$ be a $2$-colouring of $M$. If every hyperplane of $M$ contains a red element, then there exists $j \in \{0,1, \dots, k-1\}$ such that $M|B$ has a $(j+1)$-degenerate restriction with at least $|B| - b_{k-j}(|R|)$ elements. 
\end{lemma}

\begin{proof}
We proceed by induction on $k \geq 2$. For $k=2$, the result follows with $j=1$ if $B$ is collinear; otherwise, the result follows with $j=0$. Now let $k \geq 3$. Suppose first that for every line $L$ of $M|B$, there is a rank-$(k-1)$ flat $F$ skew to $L$ that contains no red points. In this case, we have $|B| \leq b_{k}(|R|)$ by definition, and so the result follows with $j=0$. Thus we may assume that there is a line $L$ of $M|B$ such that every rank-$(k-1)$ flat $F$ skew to $L$ contains a red point. Let $N$ be a simplification of the principal truncation of $L$ in $M$ such that the element corresponding to $L$ is coloured red if and only if $L \cap R \neq \emptyset$. Let $(R', B')$ be the resulting $2$-colouring of $N$. By Lemma \ref{gmep_helpful}, every hyperplane of $N$ contains a red element. By induction, there exists $j \in \{0, 1, \dots, k-2\}$ such that $X \subseteq B'$ such that $N|X$ is $(j+1)$-degenerate and $|X| \geq |B'| - b_{k-1-j}(|R'|)$. Now $X \cup L$ is $(j+2)$-degenerate in $M$. The result now follows because $|B'| + |L \del X| \geq |B|$.
\end{proof}

We apply Lemma \ref{gmep_structure} to a matroid obtained from $M$ by a series of principal truncations, where we colour the truncated points red. The truncated points together with the blue degenerate set will then form the desired degenerate set in $M$. In order to prove this, we need the following rank formula, whose easy inductive proof is omitted.

\begin{proposition}\label{formula}
For $k \geq 2$, let $M$ be a rank-$(k+1)$ matroid, and let $(F_{1}, \dots, F_{t})$ be a sequence of subsets of $M$ satisfying (D2). Let $M_{1} = M$, and let $M_{j} = M \div F_{1} \div \dots \div F_{j-1}$ for each $j \in \{2, \dots, t+1\}$. For any flat $H$ of $M_{t+1}$,
\begin{align*}
    r_{M}(H) = r_{M_{t+1}}(H) + \sum_{i \in \cI} (r_{M_{i}}(F_{i}) - 1),
\end{align*}
where $\cI = \{ i \in \{1, \dots, t\} : F_{i} \subseteq H \}$. 
\end{proposition}

Now let us derive the Main Theorem from Theorem \ref{tech_main}.

\begin{proof}[Proof of the Main Theorem.]
Let $C_{k} = \max \{ 2^{k(k-1)}/\rho_{k}(0) + k, n_{k}(0) \}$, where $n_{k}$ and $\rho_{k}$ are as in Theorem \ref{tech_main}. We may assume that $M$ is simple and not $k$-degenerate, and that $|E(M)| \geq n_{k}(0)$. Now let $R = \emptyset$ and $B = E(M)$. By Theorems \ref{aggregate} and \ref{tech_main}, we may assume that $(M,R,B)$ is not $(k,0)$-degenerate, for otherwise the average hyperplane-size is at most $C_{k}$. So there is a sequence $(F_{1}, \dots, F_{t})$ of disjoint subsets of $E(M)$ satisfying (D1), (D2) and (D3). For each $i \in \{1, \dots, t\}$, take $S_{i} \subseteq F_{i}$ with $r_{M}(S_{i})=2$ and define $L_{i} = \cl_{M}(S_{i})$. Let $M_{t+1} = M \div F_{1} \div \dots \div F_{t}$. Then every hyperplane of $M_{t+1}$ contains $F_{i}$ for some $i \in \{1, \dots, t\}$. By Proposition \ref{helpful}(4), every hyperplane of $M$ contains one of the sets $F_{1}, \dots, F_{t}$. This proves that either (1) or (2) holds. 

Let $\cP = \{ \cl_{M_{t+1}}(F_{i}) : i \in \{1, \dots, t\} \}$ be the collection of truncated points in $M_{t+1}$. (Note that, for $i \neq j$, it is possible that $\cl_{M_{t+1}}(F_{i}) = \cl_{M_{t+1}}(F_{j})$.) Let $N$ be a simplification of $M_{t+1}$, and let $R$ be the set of elements of $N$ contained in a point in $\cP$. Let $B = E(N) \del R$. (Note also that $B = E(M) \del R$.) Then $(R, B)$ is a $2$-colouring of $N$ such that every hyperplane contains a red element. By Lemma \ref{gmep_structure}, there exists $h \in \{0, 1, \dots, r(N)-2\}$ such that $N|B$ has an $(h+1)$-degenerate set $X$ such that $|X| \geq |B| - b_{r(N)-1-h}(|R|)$. Let $(G_{1}, \dots, G_{s})$ be a collection of pairwise-skew flats of $N|B$, each having rank at least two, such that $X \subseteq G_{1} \cup \dots \cup G_{s}$ and $\sum_{i=1}^{s}(r_{N}(G_{i})-1) \leq h$. Now let $j = r(M)+h-r(N)$, and observe that $b_{k-j}(|R|) \leq b_{k-j}(j)$, since $|R| \leq r(M) - r(N) \leq j$. Therefore the set $X \cup (\bigcup_{P \in \cP} P)$ has size at least $|E(M)| - b_{k-j}(j)$. It remains to show that this set is $(j+1)$-degenerate in $M$. 

Because the sets $G_{1}, \dots, G_{s}$ are pairwise-skew in $M_{t+1}$, any point of $\cP$ is spanned by at most one of them in $M_{t+1}$. Let $\cP'$ be the collection of points in $\cP$ not spanned by any of $G_{1}, \dots, G_{s}$ in $M_{t+1}$. Now by Proposition \ref{formula},
\begin{align*}
    \sum_{P \in \cP'} (r_{M}(P) - 1) + \sum_{i=1}^{s}(r_{M}(\cl_{M_{t+1}}(G_{i}))-1)
    = \sum_{i=1}^{t} (r_{M_{i}}(F_{i}) - 1) + \sum_{j=1}^{s}(r_{N}(G_{j}) - 1) \leq j,
\end{align*}
as required.
\end{proof}

Now we finish the proof.

\begin{proof}[Proof of Theorem \ref{tech_main}.]
We proceed by induction on $k \geq 2$. For the base case, assume that $|B| \geq \max \{ 4\ell, \ell \gamma_{2}(\frac{1}{4})^{-1} \}$. Suppose first that $|X_{2}| < \frac{3}{4}|B|$. In this case, there are at least $\gamma_{2}(\frac{1}{4})|B|^{2}$ lines in $M|B$. Since each red element is spanned by at most $\frac{1}{2}|B|$ lines of $M|B$, there are at most $\frac{1}{2}\ell|B|$ lines of $M|B$ which span a red element in $M$. Hence the number of monochromatic-blue lines in $M$ is at least $(\gamma_{2}(\frac{1}{4}) - \frac{\ell}{2|B|})|B|^{2} \geq \frac{1}{2}\gamma_{2}(\frac{1}{4})|B|^{2}$. Thus we may assume that $|X_{2}| \geq \frac{3}{4}|B|$. For each $e \in B \del X_{2}$, there are at least $|X_{2}| - (|B|-|X_{2}|+ \ell)$ (monochromatic-blue) lines in $M$ which contain only $e$ and an element of $X_{2}$. Hence the number of monochromatic-blue lines in $M$ is at least 
\begin{align*}
    (2|X_{2}| - |B| - \ell)(|B| - |X_{2}|) &\geq (\frac{1}{2}|B| - \ell)(|B|-|X_{2}|) \\
    &\geq \frac{1}{4}|B|(|B|-|X_{2}|),
\end{align*}
as required.

Now let $k \geq 3$ and $\ell \geq 0$. Let $\cH^{b}$ be the collection of monochromatic-blue hyperplanes of $M$. Suppose first that $|X_{k}| < \frac{1}{2}|B|$. In this case, $M|B$ has at least $\gamma_{k}(\frac{1}{2})|B|^{k}$ hyperplanes by Do's Theorem. Since each element of $R$ is spanned by at most $|B|^{k-1}$ hyperplanes of $M|B$, we have that
\begin{align*}
    |\mathcal{H}^{b}| \geq \Big( \gamma_{k}(1/2) - \frac{\ell}{|B|} \Big) |B|^{k} \geq \frac{1}{2}\gamma_{k}(1/2)|B|^{k},
\end{align*}
provided that $|B|$ is sufficiently large as a function of $k$ and $\ell$. 

Now assume that $|X_{k}| \geq \frac{1}{2}|B|$. Choose $i \in \{2, \dots, k\}$ as small as possible such that $|X_{i}| \geq |B|(2k)^{-(1+2(k-i))}$ and $|X_{i-1}| < |B|(2k)^{-(1+2(k-i+1))}$. (Recall that $X_{1}$ denotes the empty set.)

\begin{claim}\label{iflat}
$X_{i}$ is a rank-$i$ flat in $M|B$.
\end{claim}

\begin{proof}[Proof of claim.]
Suppose otherwise. Since $X_{i}$ is $i$-degenerate, there is a collection $(F_{1}, \dots, F_{t})$, where $t \geq 2$, of pairwise-skew flats in $M$, each with rank at least two, such that $X_{i} = F_{1} \cup \dots \cup F_{t}$, and $\sum_{j=1}^{t} (r_{M}(F_{j})-1) \leq i-1$. Then there exists $j \in \{1, \dots, t\}$ such that $|F_{j}| \geq \frac{|X_{i}|}{t} \geq \frac{|X_{i}|}{k} > |X_{i-1}|$. But, since $t \geq 2$, $r_{M}(F_{j}) \leq i-1$, that is, $F_{j}$ is an $(i-1)$-degenerate set, contrary to the choice of $X_{i-1}$.
\end{proof}

By Claim \ref{iflat} and Do's Theorem, we have that
\begin{equation}\label{Xi_hyp}
    |\cF_{i-1}(M|X_{i})| \geq \gamma_{i-1}(1-(2k)^{-2})|X_{i}|^{i-1} \geq C_{1}|B|^{i-1}
\end{equation}
for some constant $C_{1} > 0$ depending only on $k$. 

\begin{claim}\label{hyp}
If $i=k$, then $|\cH^{b}| \geq \frac{C_{1}^{2}}{128}|B|^{k-1}(|B|-|X_{k}|)$.
\end{claim}

\begin{proof}[Proof of claim.]
Let $e \in B \setminus X_{k}$, and let $\mathcal{F}_{e}$ denote the collection of hyperplanes in $\mathcal{H}^{b}$ spanned by $e$ and a hyperplane of $M|X_{k}$. Since each element of $R$ is spanned by at most $|B|^{k-2}$ hyperplanes of $M|X_{k}$ in $M / e$, we have 
\begin{equation}\label{f_e}
    |\mathcal{F}_{e}| \geq C_{1} |B|^{k-1} - \ell |B|^{k-2} \geq \frac{C_{1}}{2}|B|^{k-1},
\end{equation}
provided that $|B|$ is sufficiently large as a function of $k$ and $\ell$. Similarly, since each element of $B \setminus X_{k}$ is spanned by at most $|B|^{k-2}$ hyperplanes of $M|X_{k}$ in $M / e$, we have that
\begin{equation}\label{out_agg}
    \sum_{F \in \mathcal{F}_{e}} |F \setminus X_{k}| \leq |B|^{k-2}(|B| - |X_{k}|) + |B|^{k-1} \leq 2|B|^{k-1}.
\end{equation}
Let $\mathcal{F}_{e}'$ be the collection of flats $F \in \mathcal{F}_{e}$ such that $|F \del X_{k}| \leq \frac{8}{C_{1}}$. Combining (\ref{f_e}) and (\ref{out_agg}) and applying Markov's Inequality, we have that $|\mathcal{F}_{e}'| \geq \frac{C_{1}}{4}|B|^{k-1}$. Hence
\begin{align*}
    \frac{C_{1}}{4} |B|^{k-1}(|B|-|X_{k}|) \leq \sum_{e \in E(M) \del X_{k}} \sum_{F \in \cF_{e}'} 1 \leq \frac{8}{C_{1}} |\mathcal{H}^{b}|,
\end{align*}
as desired.
\end{proof}

By Claim \ref{hyp}, we may assume that $i < k$. Let $M'$ be a simplification of $M \div X_{i}$, and let $x_{i}$ denote the element of $E(M') \cap X_{i}$. Let $(R',B')$ be the $2$-colouring of $M'$ given by $R' = (R \del X_{i}) \cup \{x_{i}\}$ and $B' = B \del X_{i}$, and let $\ell' = |R'|$. Then $\ell' \leq \ell+1$, and $(M', R', B')$ is not $(k-i+1, \ell')$-degenerate. Furthermore, by Lemma \ref{truncation_lemma}, $(X_{i+1} \del X_{i}, \dots, X_{k} \del X_{i})$ is an optimal $(k+1-i)$-stratification of $M'|B'$. 

It remains to consider two cases.

{\flushleft \textbf{Case 1.} $|B|-|X_{i}| < n_{k-i+1}(\ell')$. }

Since $(M, R, B)$ is not $(k,\ell)$-degenerate, there is a hyperplane $H$ of $M'$ disjoint from $R'$. By Proposition \ref{helpful}(3), $H$ is a monochromatic-blue rank-$(k-i+1)$ flat in $M$ which is skew to $X_{i}$. Since $M|X_{i} = (M \con H)|X_{i}$, it follows by (\ref{Xi_hyp}) that $M \con H$ has at least $C_{1}|B|^{i-1}$ hyperplanes. Now, because each element of $R$ is spanned by at most $|B|^{i-2}$ hyperplanes of $M|X_{i}$ in $M \con H$, we have that
\begin{align*}
    |\cH^{b}| \geq C_{1}|B|^{i-1} - \ell |B|^{i-2} \geq \frac{C_{1}}{2} |B|^{i-1},
\end{align*}
provided that $|B|$ is sufficiently large as a function of $k$ and $\ell$. The result now follows because $|B|-|X_{j}| < n_{k+1-i}(\ell')$ for each $j \in \{i, \dots, k\}$.

{\flushleft \textbf{Case 2.} $|B|-|X_{i}| \geq n_{k-i+1}(\ell')$.}

By induction, $M'$ has $\rho_{k-i+1}(\ell') \prod_{j=i}^{k}(|B|-|X_{j}|)$ hyperplanes which are disjoint from $R'$. Let $\cH_{\div}$ denote this collection of hyperplanes. By Proposition \ref{helpful}(3), each hyperplane in $\cH_{\div}$ is a monochromatic-blue rank-$(k-i+1)$ flat in $M$ which is skew to $X_{i}$. By Theorem \ref{aggregate}, the average size of a hyperplane in $\cH_{\div}$ at most $C_{2}$, where $C_{2}>0$ is a constant depending only on $k$. Therefore, letting $\cH_{\div}' = \{H \in \cH_{\div} : |H| \leq 2C_{2} \}$, we have by Markov's Inequality that $|\cH_{\div}'| \geq \frac{1}{2}|\cH_{\div}|$. 

For $H \in \cH_{\div}'$, let $\cF_{H}$ denote the collection of monochromatic-blue hyperplanes of $M$ spanned by $H$ and a hyperplane of $M|X_{i}$. Now $M|X_{i} = (M \con H)|X_{i}$, and each element of $R$ is spanned by at most $|B|^{i-2}$ hyperplanes of $M|X_{i}$ in $M \con H$. Hence, by (\ref{Xi_hyp}),  
\begin{equation}\label{H_hyp}
    |\cF_{H}| \geq C_{1} |B|^{i-1} - \ell |B|^{i-2} \geq \frac{C_{1}}{2} |B|^{i-1},
\end{equation}
provided that $|B|$ is sufficiently large as a function of $k$ and $\ell$. Now observe that
\begin{align*}
    \frac{1}{|\mathcal{F}_{H}|} \sum_{F \in \mathcal{F}_{H}} |F \setminus X_{i}| \leq
    \frac{1}{|\mathcal{F}_{H}|} (|B|^{i-1} + |\mathcal{F}_{H}||H| ) \leq \frac{2}{C_{1}} + 2C_{2}.
\end{align*}
Let $\cF_{H}' = \{F \in \cF_{H} : |F \del X_{i}| \leq 4/C_{1} + 4C_{2} \}$. By Markov's Inequality and (\ref{H_hyp}), $|\cF_{H}'| \geq \frac{C_{1}}{4}|B|^{i-1}$. Now, because every flat of $\cH_{\div}$ is skew to $X_{i}$ in $M$, each hyperplane in $\cF_{H}'$ contains at most $(4/C_{1} + 4C_{2})^{k+1-i}$ flats in $\cH_{\div}$. Therefore
\begin{align*}
    \frac{C_{1}}{8}|B|^{i-1}|\cH_{\div}| \leq \frac{C_{1}}{4} |B|^{i-1} |\cH_{\div}'| \leq \sum_{H \in \cH_{\div}'} \sum_{F \in \cF_{H}'} 1 \leq \Big( \frac{4}{C_{1}} + 4C_{2} \Big)^{k+1-i}|\cH^{b}|.
\end{align*}
The result now follows, because $|\cH_{\div}| \geq \rho_{k-i+1}(\ell')\prod_{j=i}^{k}(|B|-|X_{j}|)$.
\end{proof}

\section*{Acknowledgements}

We thank Jim Geelen for helpful discussions on the topic of this paper. The second author is grateful to IBS DIMAG for hosting him as a visitor in January and July-August 2024, during which time much of the research for this paper was conducted.

\bibliographystyle{plain}
\bibliography{CKL}

\begin{thebibliography}{10}

\bibitem{ackerman2008there}
Eyal Ackerman, Kevin Buchin, Christian Knauer, Rom Pinchasi, and G{\"u}nter
  Rote.
\newblock There are not too many magic configurations.
\newblock {\em Discrete \& Computational Geometry}, 39(1-3):3--16, 2008.

\bibitem{Beck}
J{\'o}zsef Beck.
\newblock On the lattice property of the plane and some problems of {D}irac,
  {M}otzkin and {E}rd{\H{o}}s in combinatorial geometry.
\newblock {\em Combinatorica}, 3:281--297, 1983.

\bibitem{CGK}
Rutger Campbell, Jim Geelen, and Matthew Kroeker.
\newblock Average plane-size in complex-representable matroids.
\newblock {\em arXiv preprint arXiv:2310.02826}, 2023.

\bibitem{chakerian1970sylvester}
G.~Donald Chakerian.
\newblock Sylvester's problem on collinear points and a relative.
\newblock {\em The American Mathematical Monthly}, 77(2):164--167, 1970.

\bibitem{Do}
Thao Do.
\newblock Extending {E}rd{\H{o}}s--{B}eck's theorem to higher dimensions.
\newblock {\em Computational Geometry}, 90:101625, 2020.

\bibitem{dvir2015quantitative}
Zeev Dvir and Christian Tessier-Lavigne.
\newblock A quantitative variant of the multi-colored {M}otzkin--{R}abin
  theorem.
\newblock {\em Discrete \& Computational Geometry}, 53(1):38--47, 2015.

\bibitem{ErdosPurdy}
Paul Erd{\H{o}}s and George Purdy.
\newblock Some combinatorial problems in the plane.
\newblock {\em J. Combin. Theory Ser. A}, 25:205--210, 1978.

\bibitem{Grunbaum}
Branko Gr\"unbaum.
\newblock Arrangements of colored lines.
\newblock {\em Notices Amer. Math. Soc.}, 22:A--200, 1975.

\bibitem{grunbaum1999monochromatic}
Branko Gr{\"u}nbaum.
\newblock Monochromatic intersection points in families of colored lines.
\newblock {\em Geombinatorics}, 9(1):3--9, 1999.

\bibitem{han2017note}
Zeye Han.
\newblock A note on the weak {D}irac conjecture.
\newblock {\em The Electronic Journal of Combinatorics}, 24(1):P1--63, 2017.

\bibitem{hirzebruch1983arrangements}
Friedrich Hirzebruch.
\newblock Arrangements of lines and algebraic surfaces.
\newblock In {\em Arithmetic and geometry, Vol. 2}, pages 113--140.
  Birkh{\"a}user, 1983.

\bibitem{huicochea2021number}
Mario Huicochea.
\newblock On the number of monochromatic lines in {$\mathbb{R}^d$}.
\newblock {\em Discrete \& Computational Geometry}, 65:1061--1086, 2021.

\bibitem{huicochea2022note}
Mario Huicochea, Jes{\'u}s Lea{\~n}os, and Luis~Manuel Rivera.
\newblock A note on the minimum number of red lines needed to pierce the
  intersections of blue lines.
\newblock {\em Computational Geometry}, 104:101863, 2022.

\bibitem{jamison1983catalogue}
Robert Jamison and Dale Hill.
\newblock A catalogue of sporadic slope-critical configurations.
\newblock In {\em Proceedings of the Fourteenth Southeastern Conference on
  Combinatorics, Graph Theory and Computing (Boca Raton, Fla., 1983), Congr.
  Numer}, volume~40, pages 101--125, 1983.

\bibitem{langer2003logarithmic}
Adrian Langer.
\newblock Logarithmic orbifold euler numbers of surfaces with applications.
\newblock {\em Proceedings of the London mathematical society}, 86(2):358--396,
  2003.

\bibitem{Lovasz}
L{\'a}szl{\'o} Lov{\'a}sz.
\newblock Flats in matroids and geometric graphs.
\newblock In {\em Combinatorial Surveys (Proc. 6th British Combinatorial
  Conference}, pages 45--86, 1977.

\bibitem{Lund}
Ben Lund.
\newblock Essential dimension and the flats spanned by a point set.
\newblock {\em Combinatorica}, 38(5):1149--1174, 2018.

\bibitem{lund2021two}
Ben Lund.
\newblock Two theorems on point-flat incidences.
\newblock {\em Computational Geometry}, 92:101681, 2021.

\bibitem{lund2014pseudoline}
Ben Lund, George~B Purdy, and Justin~W Smith.
\newblock A pseudoline counterexample to the {S}trong {D}irac {C}onjecture.
\newblock {\em The Electronic Journal of Combinatorics}, pages P2--31, 2014.

\bibitem{Melchior}
Eberhard Melchior.
\newblock {\"U}ber {V}ielseite der projektiven {E}bene.
\newblock {\em Deutsche Math.}, 5:461--475, 1941.

\bibitem{motzkin1951lines}
Theodore Motzkin.
\newblock The lines and planes connecting the points of a finite set.
\newblock {\em Transactions of the American Mathematical Society},
  70(3):451--464, 1951.

\bibitem{motzkin1967nonmixed}
Theodore Motzkin.
\newblock Nonmixed connecting lines.
\newblock {\em Notices Amer. Math. Soc}, 14:837, 1967.

\bibitem{oxley2006matroid}
James Oxley.
\newblock {\em Matroid theory (second edition)}, volume~21.
\newblock Oxford University Press, USA, 2011.

\bibitem{pinchasi2013solution}
Rom Pinchasi.
\newblock A solution to a problem of {G}r{\"u}nbaum and {M}otzkin and of
  {E}rd{\H{o}}s and {P}urdy about bichromatic configurations of points in the
  plane.
\newblock {\em Israel Journal of Mathematics}, 198:205--214, 2013.

\bibitem{pinchasi2021theorem}
Rom Pinchasi.
\newblock A theorem about vectors in {$R^2$} and an algebraic proof of a
  conjecture of {E}rd{\H{o}}s and {P}urdy.
\newblock {\em AUSTRALASIAN JOURNAL OF COMBINATORICS}, 81(1):170--186, 2021.

\bibitem{pinchasi2020one}
Rom Pinchasi and Alexander Polyanskii.
\newblock A one-page solution of a problem of {E}rd{\H{o}}s and {P}urdy.
\newblock {\em Discrete \& Computational Geometry}, 64(2):382--385, 2020.

\bibitem{pretorius2009sylvester}
Lou~M Pretorius and Konrad~J Swanepoel.
\newblock The {S}ylvester--{G}allai theorem, colourings and algebra.
\newblock {\em Discrete Mathematics}, 309(2):385--399, 2009.

\bibitem{scott1970sets}
Paul Scott.
\newblock On the sets of directions determined by n points.
\newblock {\em The American Mathematical Monthly}, 77(5):502--505, 1970.

\bibitem{szekely1997crossing}
L{\'a}szl{\'o} Sz{\'e}kely.
\newblock Crossing numbers and hard {E}rd{\H{o}}s problems in discrete
  geometry.
\newblock {\em Combinatorics, Probability and Computing}, 6(3):353--358, 1997.

\bibitem{SzemerediTrotter}
Endre Szemer{\'e}di and William Trotter.
\newblock Extremal problems in discrete geometry.
\newblock {\em Combinatorica}, 3:381--392, 1983.

\bibitem{Toth}
Csaba T{\'o}th.
\newblock The {S}zemer{\'e}di-{T}rotter theorem in the complex plane.
\newblock {\em Combinatorica}, 35:95--126, 2015.

\bibitem{ungar1982directions}
Peter Ungar.
\newblock 2n noncollinear points determine at least 2n directions.
\newblock {\em Journal of Combinatorial Theory, Series A}, 33(3):343--347,
  1982.

\bibitem{Zahl}
Joshua Zahl.
\newblock A {S}zemer{\'e}di--{T}rotter type theorem in {$R^4$}.
\newblock {\em Discrete \& Computational Geometry}, 54:513--572, 2015.

\end{thebibliography}

\end{document}